\documentclass[a4paper,12pt]{article}
\usepackage{amsmath,amsthm,amssymb,latexsym,txfonts}

\usepackage[usenames]{color}
\usepackage[ansinew]{inputenc}

\definecolor{rojo}{rgb}{1,0,0}
\definecolor{blanco}{rgb}{1,1,1}

\newcommand{\bigint}{\begin{picture}(10,10)
\put(-1,2){\line(1,0){10}}
\end{picture}\kern-14pt\int}

\title{{\bf The maximal singular integral: estimates in terms of the singular integral \\
}}
\author{\Large{\Large Joan Verdera}}
\newtheorem{teo}{Theorem}

\newtheorem{lemma}[teo]{Lemma}

\newtheorem*{example}{Example}

\theoremstyle{definition}
\newtheorem*{gracies}{Acknowledgements}

\newcommand{\C}{\mathbb{C}}
\newcommand{\R}{\mathbb R}

\begin{document}
\date{}

\maketitle

\begin{abstract}
This paper considers estimates of the maximal singular integral
$T^*f$ in terms of the singular integral $Tf$ only. The most basic
instance of the estimates we look for is the $L^2(\R^n)$ inequality
$\|T^*f\|_2 \le C\, \|Tf\|_2 $. We present the complete
characterization, recently obtained by Mateu, Orobitg, P\'{e}rez and
the author, of the smooth homogeneous convolution
Calder\'{o}n--Zygmund operators for which such inequality holds. We
focus attention on special cases of the general statement to convey
the main ideas of the proofs in a transparent way, as free as
possible of the technical complications inherent to the general
case. Particular attention is devoted to higher Riesz transforms.
\end{abstract}

\section{Introduction} \label{Verdera:Sec:1}

In this expository paper we consider the problem of estimating the
Maximal Singular Integral $T^*f$ only in terms of the Singular
Integral $Tf.$  In other words, the function $f$ should appear in
the estimates only through $Tf.$ The context is that of classical
Calder\'{o}n--Zygmund theory: we deal with smooth homogeneous
convolution singular integral operators of the type
\begin{equation}\label{SI}
Tf(x)= p.v. \int f(x-y)\,K(y)\,dy \equiv \lim_{\epsilon\rightarrow
0} T^\epsilon f(x)\,,
\end{equation}
where
$$
T^{\epsilon}f(x)= \int_{| y-x| > \epsilon} f(x-y) K(y) \,dy
$$
is the truncated integral at level $\epsilon$. The kernel $K$ is

\begin{equation}\label{K}
K(x)=\frac{\Omega(x)}{|x|^n},\quad x \in \mathbb{R}^n \setminus
\{0\}\,,
\end{equation}
where $\Omega$ is a (real valued) homogeneous function of degree $0$
whose restriction to the unit sphere $S^{n-1}$ is of
class~$C^\infty(S^{n-1})$ and satisfies the cancellation property
\begin{equation}\label{can}
\int_{|x|=1} \Omega(x)\,d\sigma(x)=0\,,
\end{equation}
$\sigma$ being the normalized surface measure on $S^{n-1}$. The
maximal singular integral is
$$
T^{\star}f(x)= \sup_{\epsilon > 0} | T^{\epsilon}f(x)|, \quad x \in
\mathbb{R}^n \,.
$$
As we said before, the problem we are envisaging consists in
estimating $T^{\star}f$ in terms of $Tf$ only. The well known
Cotlar's inequality
\begin{equation}\label{MS}
T^{\star}f(x) \leqslant C  \left(M(Tf)(x) + Mf (x)\right),\quad x
\in \mathbb{R}^n \,,
\end{equation}
is of no use because it contains the term $f$ besides $Tf$. The most
basic form of the estimate we are looking for is the $L^2$
inequality
\begin{equation} \label{L2}
 \| T^{\star} f  \|_2 \leqslant  C \| T f \|_2,\quad f \in L^2(\mathbb{R}^n) \,.
\end{equation}
This problem arose when the author was working at the David--Semmes
problem (\cite[p.139, first paragraph]{DS}). It was soon discovered
(\cite{MV}) that the parity of the kernel plays an essential role.
Some years after, a complete characterization of the even operators
for which \eqref{L2} holds was presented in \cite{MOV} and
afterwards the case of odd kernels was solved in \cite{MOPV}.
Unfortunately there does not seem to be a way of adapting the
techniques of those papers to the Ahlfors regular context in which
the David--Semmes problem was formulated.

The proof of the main result in \cite{MOV} and \cite{MOPV} is long
and technically involved. It is the purpose of this paper to
describe the main steps of the argument in the most transparent way
possible. We give complete proofs of particular instances of the
main results of the papers mentioned, so that the reader may grasp,
in a simple situation, the idea behind the proof of the general
cases. Thus, in a sense, the present paper could serve as an
introduction to \cite{MOV} and \cite{MOPV}.

Notice that \eqref{L2} is true whenever $T$ is a continuous
isomorphism of $L^2(\mathbb{R}^n)$ onto itself. Indeed a classical
estimate, which follows from Cotlar's inequality, states that
\begin{equation} \label{Is}
 \| T^{\star} f  \|_2 \leqslant  C\, \| f \|_2,   \quad f \in L^2(\mathbb{R}^n) \,,
\end{equation}
which combined with the assumption that $T$ is an isomorphism gives
\eqref{L2}. Thus \eqref{L2} is true for the Hilbert Transform and
for the Beurling Transform. The first non-trivial case is a scalar
Riesz transform in dimension $2$ or higher. Recall that the
\emph{j}-th Riesz transform is the Calder\'{o}n--Zygmund operator
with kernel
$$
\frac{x_j}{|x|^{n+1}}, \quad x \in \mathbb{R}^n \setminus \{0\},
\quad 1 \leqslant j \leqslant n .
$$

The first non trivial case for even operators is any second order
Riesz transform. For example, the second order Riesz transform with
kernel
$$
\frac{x_1 x_2}{|x|^{n+2}}, \quad x \in \mathbb{R}^n \setminus \{0\}.
$$

In Section \ref{Verdera:Sec:2} we prove the $L^2$ estimate
\eqref{L2} for the second order Riesz transform above and in Section
\ref{Verdera:Sec:4} for the $\emph{j}-th$ Riesz transform. Indeed,
in both cases we prove a stronger pointwise estimate which works for
all higher Riesz transforms. Recall that a higher Riesz transform is
a smooth homogeneous convolution singular integral operator with
kernel of the type
$$
\frac{P(x)}{|x|^{n+d}}, \quad x \in \mathbb{R}^n \setminus \{0\},
$$
where $P$ is a harmonic homogeneous polynomial of degree $d
\geqslant 1$. The mean value property of harmonic functions combined
with homogeneity yields the cancellation property \eqref{can}. One
has the following (\cite{MOV})
\begin{teo}\label{eteo}
If $T$ is an even higher Riesz transform, then
\begin{equation}\label{peven}
 T^{*}f(x) \leqslant C \, M(Tf)(x),\quad x \in \R^n\,, \quad f \in
 L^2(\R^n)\,,
\end{equation}
 where $M$ is the maximal Hardy-Littlewood operator.
\end{teo}
Indeed, for a second order Riesz transform $S$ one has that the
truncation at level $\epsilon$ is a mean of $S(f)$ on a ball. More
precisely one has
\begin{equation}\label{mean}
S^{\epsilon}(f)(x)= \frac{1}{|B(x,\epsilon)|}\int_{B(x,\,\epsilon)}
S(f)(y)\,dy
\end{equation}
A weighted variant of the preceding identity works for a general
even higher Riesz transform. Of course, \eqref{L2} for even higher
Riesz transforms follows immediately from \eqref{peven}. It turns
out that, as we explain in Section \ref{Verdera:Sec:3},
\eqref{peven} does not hold for odd Riesz transforms, not even for
the Hilbert transform. But we can prove the following substitute
result (\cite{MOPV}), which obviously takes care of \eqref{L2} for
odd higher Riesz transforms.
\begin{teo}\label{oteo}
If $T$ is an odd higher Riesz transform, then
\begin{equation}\label{podd}
 T^{*}f(x) \leqslant C \, M^2(Tf)(x),\quad x \in \R^n\,, \quad f \in
 L^2(\R^n)\,,
\end{equation}
where $M^2=M \circ M$ is the iteration of the maximal Hardy-
Littlewood operator.
\end{teo}

Without any harmonicity assumption the $L^2$ estimate \eqref{L2}
does not hold. The simplest example involves the Beurling transform
$B$, which is the singular integral operator in the plane with
complex valued kernel
$$
- \frac{1}{\pi}\frac{1}{z^2}= - \frac{1}{\pi}
\frac{\overline{z}^2}{|z|^4} = - \frac{1}{\pi} \frac{x^2-y^2}{|z|^4}
+ \text{i}  \frac{1}{\pi} \frac{2xy}{|z|^4}.
$$
The Fourier transform of the tempered distribution $p.v.(-
\frac{1}{\pi}\frac{1}{z^2})$ is the function
$\frac{\overline{\xi}}{\xi},$ so that $B$ is an isometry of
$L^2(\R^2)$ onto itself. It turns out that the singular integral
$$
T= B+B^2 = B(I + B)
$$
does not satisfy the $L^2$ control \eqref{L2}. The reason for that,
as we will see later on in this Section, is that the operator $I+B$
is not invertible in $L^2(\R^2).$

One way to explain the difference between the even and odd cases is
as follows. Theorem \ref{eteo} concerns an even higher Riesz
transform determined by a harmonic homogeneous polynomial of degree,
say,  $d$. In its proof one is lead to consider the operator
$(-\bigtriangleup)^{d/2}$, which is a differential operator.
Instead, in Theorem \ref{oteo},  $d$ is odd and thus
$(-\bigtriangleup)^{d/2}$ is only a pseudo-differential operator.
The effect of this is that in the odd case certain functions are not
compactly supported and are not bounded. Nevertheless, they still
satisfy a $BMO$ condition, which is the key fact in obtaining the
second iteration of the maximal operator.

The search for a description of those singular integrals $T$ of a
given parity for which \eqref{L2} holds begun just after \cite{MV}
was published. The final answer was given in \cite{MOV} and
\cite{MOPV}. To state the result denote by $A$ the
Calder\'{o}n--Zygmund algebra consisting of the operators of the
form $\lambda I + T$, where $T$ is a smooth homogeneous convolution
singular integral operator and $\lambda$ a real number.

\begin{teo}\label{MOV}
Let $T$ be an even smooth homogeneous convolution singular integral
operator with kernel  $\Omega(x)/|x|^n$. Then the following are
equivalent.
\begin{enumerate}
\item [(i)]
$$T^* f(x)\leqslant C\, M(Tf)(x), \quad x\in \R^n , \quad f \in L^2(\R^n),$$
where $M$ is the Hardy-Littlewood maximal operator.
\item [(ii)]
$$\int |T^*f |^2 \leqslant C \int |Tf|^2, \quad f \in L^2(\R^n).$$

\item [(iii)]
If the spherical harmonics expansion of $\Omega$ is
$$\Omega(x) =P_{2}(x)+P_{4}(x)+\dotsb, \quad |x|=1,$$
then there exist an even harmonic homogeneous polynomial $P$ of
degree $d$, such that $P$ divides $P_{2j}$ (in the ring of all
polynomials in $n$ variables with real coeficients) for all~$j$,
$T=R_{P}\circ U$, where $R_P$ is the higher Riesz transform with
kernel $P(x)/|x|^{n+d}$, and $U$ is an invertible operator in the
Calder\'{o}n--Zygmund algebra $A$.
\end{enumerate}
\end{teo}

Several remarks are in order. First, it is surprising that the $L^2$
control we are looking for, that is, condition (ii) above, is
equivalent to the apparently much stronger pointwise inequality (i).
We do not know any proof of this fact which does not go through the
structural condition (iii). Second, condition (iii) on the
 spherical harmonics expansion of $\Omega$ is purely algebraic and
 easy to check in practice on the Fourier transform side.
 Observe that if condition (iii) is satisfied, then the polynomial $P$
  must be a scalar multiple of the first
 non-zero spherical harmonic $P_{2j}$ in the expansion of $\Omega.$
 We illustrate this with an example.

 \pagebreak

\begin{example}
\end{example}
Let $P(x,y) = -\frac{1}{\pi} xy$ and
denote by $R_P$ the second order Riesz transform in the plane
associated with the harmonic homogeneous polynomial $P$. Its kernel
is
\begin{equation}\label{second}
-\frac{1}{\pi} \frac{xy}{|z\,|^4}, \quad z=x+i y \in
\mathbb{C}\setminus\{0\}.
\end{equation}
According to a well known formula \cite[p.73]{St} the Fourier
transform of the principal value distribution associated with this
kernel is
$$
\frac{ uv}{|\xi\,|^2}, \quad \xi =u+i v \in
\mathbb{C}\setminus\{0\}.
$$
This is also the symbol (or Fourier multiplier) of $R_P$, in the
sense that
$$
\widehat{R_P(f)}(\xi) = \frac{uv}{|\xi\,|^2} \,\hat{f}(\xi), \quad
\xi \neq 0, \quad f \in L^2(\R^n).
$$

 Similarly, the Fourier multiplier of the
fourth order Riesz transform with kernel
$$
\frac{2}{\pi} \frac{x^3 y- x y^3}{|z\,|^6}, \quad z \neq 0,
$$
is
$$
\frac{u^3 v- u v^3}{|\xi\,|^4}, \quad \xi \neq 0.
$$
Given a real number $\lambda$ let $T$ be the singular integral with
kernel
$$
-\frac{1}{\pi} \frac{2xy}{|z\,|^4}+ \lambda \, \frac{2}{\pi}
\frac{x^3 y- x y^3}{|z\,|^6}.
$$
Its symbol is
$$
\frac{uv}{|\xi\,|^2}\left( 1+\lambda \frac{u^2-v^2}{|\xi|^2}\right).
$$
 We clearly have
$$
T= R_P \circ U,
$$
$U$ being the bounded operator on $L^2(\R^n)$ with symbol $1+\lambda
\frac{u^2-v^2}{|\xi|^2}.$ Notice that the multiplier $1+\lambda
\frac{u^2-v^2}{|\xi|^2}$ vanishes at some point of the unit sphere
if and only if $|\lambda| \geqslant 1.$ Therefore condition (iii) of
Theorem \ref{MOV} is satisfied if and only if $|\lambda| < 1.$  For
instance, taking $\lambda =1$ one gets an operator for which neither
the $L^2$ estimate (ii) nor the pointwise inequality (i) hold.

To grasp the subtlety of the division condition in (iii) it is
instructive to consider the special case of the plane. The function
$\Omega$, which is real,  has a Fourier series expansion
\begin{equation*}
\begin{split} \Omega(e^{i\theta}) & =\sum_{n=-\infty}^{\infty}
c_n\,e^{i n \theta} = \sum_{n=1}^{\infty} c_n\,e^{i n \theta}+
\overline{c_n}\,e^{-i n \theta}\\*[5pt] & = \sum_{n=1}^{\infty} 2\,
\text{Re}( c_n\,e^{i n \theta})
\end{split}
\end{equation*}
The expression $2\, \text{Re}( c_n\,e^{i\, n \theta})$ is the
general form of the restriction to the unit circle of a harmonic
homogeneous polynomial of degree $n$ on the plane.  There are
exactly $2n$ zeroes of $2\, \text{Re}( c_n\,e^{i\, n \theta})$ on
the circle, which are uniformly distributed. They are the $2 n$-th
roots of unity if and only if $c_n$ is purely imaginary.

Since $\Omega$ is even, only the Fourier coefficients with even
index may be non-zero and so
$$
\Omega(e^{i\theta})= \sum_{n=1}^{\infty} 2\, \text{Re}(
c_{2\,n}\,e^{i\, 2\,n \theta}).
$$
Replacing $\theta$ by $\theta+\alpha$ we obtain
$$
\Omega(e^{i (\theta +\alpha)})= \sum_{n=N}^{\infty} 2\, \text{Re}(
c_{2\,n}\,e^{i \, 2\,n \alpha} \,e^{i \,2\,n \theta}),
$$
where $c_{2\,N} \neq  0.$ Take $\alpha$ so that $c_{2\,N}\,e^{i\,2
\,N\,\alpha}$ is purely imaginary. Set $\gamma_{2\,n}= c_{2\,n}\,
e^{i\,2 \,n\,\alpha}.$ Then
$$
\Omega(e^{i (\theta +\alpha)})= \sum_{n=N}^{\infty} 2\, \text{Re}(
\gamma_{2\,n}\,e^{i \,2\,n \theta}).
$$
If $\text{Re}( \gamma_{2\,N}\,e^{i \,2\,N \theta})$ divides
$\text{Re}( \gamma_{2\,n}\,e^{i \,2\,n \theta})$, then , for some
positive integer $k$,
$$
k \frac{\pi}{4\,n}= \frac{\pi}{4\,N},
$$
or $n=k\,N.$ This means that only the Fourier coefficients with
index a multiple of $2\,N$ may be non-zero :
$$
\Omega(e^{i (\theta +\alpha)})=  \sum_{p=1}^{\infty} 2\, \text{Re}(
\gamma_{2\,N\,p}\;e^{i \,2\,N\,p\, \theta}).
$$
Moreover $\gamma_{2Np}$ must be purely imaginary, that is,
$\gamma_{2\,N\,p} = r_{2\,N\,p}\,i$, with $r_{2\,N\,p}$ real.
Replacing $\theta+\alpha$ by $\theta$ we get
\begin{equation*}
\begin{split}
\Omega(e^{i\theta})& = \sum_{p=1}^{\infty} 2\, \text{Re}(
r_{2Np}\,i\,\,e^{-i 2Np\alpha } \,e^{i 2Np \theta}),\\*[5pt] & =
\sum_{p=1}^{\infty}   r_{2Np}\,i\,\,e^{-i 2Np\alpha} \,e^{i 2Np
\theta}- r_{2Np}\,i\,\,e^{i 2Np\alpha} \,e^{-i 2Np \theta}.
\end{split}
\end{equation*}
As it is well-known the sequence of the $r_{2Np}, p=1,2, \dots$ is
rapidly decreasing, because $\Omega(e^{i\theta})$ is infinitely
differentiable. Therefore the division property in condition (iii)
of Theorem \ref{eteo} can be reformulated as a statement about the
 arguments and the support of the Fourier coefficients
 of $\Omega(e^{i\theta}).$

For odd operators the statement of Theorem \ref{MOV} must be
slightly modified (\cite{MOPV}).

\pagebreak
\begin{teo}\label{MOPV}
Let $T$ be an odd smooth homogeneous convolution singular integral
operator with kernel  $\Omega(x)/|x|^n$. Then the following are
equivalent.
\begin{enumerate}
\item [(i)]
$$T^* f(x)\leqslant C\, M^2 (Tf)(x), \quad x\in \R^n , \quad f \in L^2(\R^n),$$
$M^2 = M \circ M$ being the iterated Hardy-Littlewood maximal
operator.
\item [(ii)]
$$\int |T^*f |^2 \leqslant C \int |Tf|^2, \quad f \in L^2(\R^n).$$

\item [(iii)]
If the spherical harmonics expansion of $\Omega$ is
$$\Omega(x) =P_{1}(x)+P_{3}(x)+\dotsb, \quad |x|=1,$$
then there exist an odd harmonic homogeneous polynomial $P$ of
degree $d$, such that $P$ divides $P_{2j+1}$ (in the ring of all
polynomials in $n$ variables with real coeficients) for all~$j$,
$T=R_{P}\circ U$, where $R_P$ is the higher Riesz transform with
kernel $P(x)/|x|^{n+d}$, and $U$ is an invertible operator in the
Calder\'{o}n--Zygmund algebra $A$.
\end{enumerate}
\end{teo}

Sections \ref{Verdera:Sec:2} and \ref{Verdera:Sec:4} contain ,
respectively, the proofs of Theorems \ref{eteo} and  \ref{oteo} for
the most simple kernels. In Section \ref{Verdera:Sec:3} we show that
the Hilbert transform does not satisfy the pointwise inequality
\eqref{peven}. In Section \ref{Verdera:Sec:5} we prove that
condition (iii) in Theorem \ref{MOV} is necessary and in Section
\ref{Verdera:Sec:6} that it is sufficient,  in both cases in
particularly simple situations. Section \ref{Verdera:Sec:7} contains
brief comments on the proof of the general case and a mention of a
couple of open problems.

\section{Proof of Theorem \ref{eteo} for second order Riesz transforms.}
\label{Verdera:Sec:2} For se sake of clarity we work only with the
second order Riesz transform $T$ with kernel
$$\frac{x_{1}x_{2}}{|x|^{n+2}} .$$

The inequality to be proven, namely \eqref{peven}, is invariant by
translations and by dilations, so that we only need to show that

\begin{equation}\label{po}
|T^{1}f(0)|\le C\, M(Tf)(0),
\end{equation}
where
$$T^1 f(0) = \int_{\R^n \setminus B}
\frac{x_{1}x_{2}}{|x|^{n+2}}f(x)\,dx$$ is the truncation at level
$1$ at the origin. Here $B$ is the unit (closed) ball centered at
the origin. A natural way to show \eqref{po} is to find a function
$b$ such that
$$
\chi_{\R^n \setminus B}(x)\,\frac{x_{1}x_{2}}{|x|^{n+2}} = T(b).
$$
One should keep in mind that $T$ is injective but not onto. Then
there is no reason whatsoever for such a $b$ to exist. If such a $b$
exists then
\begin{equation}\label{trunc}
T^1 f(0) = \int\!\! Tb(x) \; f(x)\,dx = \int b(x)\;T(f)(x)\,dx
\end{equation}
If moreover $b$ is in $L^\infty(\R^n)$ and is supported on $B$, we
get
$$
|T_{1}f(0)|\le \|b\|_{\infty}|B| \frac{1}{|B|} \int_B |T(f)(x)|\,dx
\le C M(T(f))(0).
$$

Thus everything has been reduced to the following lemma.

\begin{lemma}\label{nuc}
There exists a bounded measurable function $b$ supported on $B$ such
that
$$
\chi_{\R^n \setminus B}(x)\,\frac{x_{1}x_{2}}{|x|^{n+2}} = T(b)(x),
\quad \text{for almost all}\quad x \in \R^n.
$$
\end{lemma}
\begin{proof}
Let $E$ be the standard fundamental solution of the Laplacian in
$\R^n$. Then, for some dimensional constant $c_n,$ we have that, in
the distributions sense,
\begin{equation}
\partial_1 \partial_2 E = c_n \; p.v. \frac{x_{1}x_{2}}{|x|^{n+2}}.
\end{equation}
Let us define a function $\varphi$ by
\begin{equation}\label{fi}
\varphi(x)=\begin{cases}
E(x)&\text{on }\R^n \setminus B\\
A_{0}+A_{1}|x|^2&\text{on }B
\end{cases}
\end{equation}
where the constants $A_0$ and $A_1$ are chosen so that $\varphi$ and
$\nabla \varphi$ are continuous on $\R^n.$  This is possible
because, for each $i,$
\begin{equation*}
\partial_i \varphi (x) = \begin{cases}
c_n\;\frac{x_{i}}{|x|^{n}}, &  \quad x \in \R^n \setminus B \\
2 A_1 x_i, &  \quad x \in B
\end{cases}
\end{equation*}
and so, for an appropriate choice of $A_1$, the above two
expressions coincide on $\partial B$ for all $i,$ or, equivalently,
$\nabla\varphi$ is continuous. The continuity of $\varphi$ is now
just a matter of choosing $A_0$ so that $E(x) = A_0 +A_1 |x|^2$ on
$\partial B,$ which is possible because $E$ is radial.

The continuity of $\varphi$ and $\nabla\varphi$ guaranties that we
can compute a second order derivative of $\varphi$ in the
distributions sense by just computing it pointwise on $B$ and on
$\R^n \setminus B.$ The reason is that no boundary terms will appear
when applying Green-Stokes to compute the action of the second order
derivative of $\varphi$ under consideration on a test function.
Therefore
$$
\Delta \varphi = 2 n A_1 \chi_B \equiv b,
$$
where the last identity is the definition of $b.$ Since $\varphi = E
* \Delta\varphi $ we obtain, for some dimensional constant $c_n,$
$$
\partial_1\partial_2 \varphi = \partial_1\partial_2 E *
\Delta\varphi = c_n\;p.v. \frac{x_{1}x_{2}}{|x|^{n+2}}*
\Delta\varphi = c_n\; T(b).
$$
On the other hand, by \eqref{fi} and noticing that
$\partial_1\partial_2 |x|^2 = 0$, we get
$$
\partial_1\partial_2 \varphi = \chi_{\R^n \setminus B} (x) c_n\,
\frac{x_{1}x_{2}}{|x|^{n+2}},
$$
and the proof of Lemma \ref{nuc} is complete. 
\end{proof}

Notice that \eqref{trunc} together with the special form of the
function $b$ found in the proof of Lemma \ref{nuc} yield the formula
\eqref{mean}, namely, that a truncation at level $\epsilon$ at the
point $x$ of $S(f)$, $S$ being a second order Riesz transform, is
the mean of $S(f)$ on the ball $B(x,\epsilon)$ .

\section{The pointwise control of $T^*$ by $M\circ T$ fails for the Hilbert transform}
\label{Verdera:Sec:3}

We show now that the inequality
\begin{equation}\label{pointwiseH}
H^*f(x)\le C\, M(Hf)(x)\,, \quad x \in \R\, \quad f \in L^2(\R),
\end{equation}
where $H$ is the Hilbert transform, fails. Replacing $f$ by $H(f)$
in \eqref{pointwiseH} and recalling that $H(Hf)=-f\,,\,\, f \in
L^2(\R)$\,, we see that \eqref{pointwiseH} is equivalent to
$$H^*(H(f))(x)\le C\, M(f)(x)\,, \quad x \in \R, \quad  f \in L^2(\R).$$

It turns out that the operator $H^* \circ H$ is not of weak type
$(1,1)$.

Let us prove that if $f=\chi_{(0,1)}$, then there are positive
constants $m$ and $C$ such that whenever $x>m$,
\begin{equation}\label{log}
 H^*(Hf)(x)\ge C\,\frac{\log x}{x}
\end{equation}
This shows that $H^* \circ H$ is not of weak type $(1,1).$ Indeed,
choosing $m>e$ if necessary, we have
$$ \sup_{\lambda>0} \lambda \, |\{ x\in \R: H^*(Hf)(x) >
\lambda \} | \ge \sup_{\lambda>0} \lambda \, |\{ x > m: \frac{\log
x}{x}
> C^{-1}\, \lambda \} |
$$
$$
= C\, \sup_{\lambda>0} \lambda \, | \{ x >m  :  \frac{\log x }{x}
> \lambda \} | \ge C\, \sup_{\lambda>0} \lambda \,(
\varphi^{-1}(\lambda) - e),
$$
where $\varphi$ is the decreasing function $\varphi :(e,\infty)
\rightarrow  (0,e^{-1})$,  given by $\varphi(x) = \frac{\log x
}{x}$. To conclude observe that the right hand side of the estimate
is unbounded as $\lambda \rightarrow 0$:
$$
\lim_{\lambda \rightarrow 0}      \lambda \varphi^{-1}(\lambda) =
\lim_{\lambda \rightarrow \infty} \varphi(\lambda) \lambda = \infty.
$$

To prove \eqref{log} we recall that for $f=\chi_{(0,1)}$
$$Hf(y)
=\log\frac{|y|}{|y-1|}.
$$

Let $m>1$ big enough to be chosen later on. Take $x>m$. By
definition of $H^*$
$$H^*(Hf)(x)
\ge   \left| \int_{ |y-x|> m+x } \frac{1}{y-x}\,
\log\frac{|y|}{|y-1|} \,dy \right|
$$
and splitting the integral in the obvious way
$$\int_{-\infty}^{-m} \frac{1}{y-x}\log\frac{-y}{-y+1} \,dy +
\int_{2x+m}^{\infty} \frac{1}{y-x}\log\frac{y}{y-1}\,dy
$$
$$=
\int_{m}^{\infty} \frac{1}{x+y}\log\frac{y+1}{y} \,dy +
\int_{2x+m}^{\infty} \frac{1}{y-x}\log\frac{y}{y-1}\,dy =A(x)+B(x),
$$
where  both $A(x)$ and $B(x)$ are positive.  Hence
$$H^*(Hf)(x)
\ge A(x).
$$
Since $$\log(1+\frac1y) \approx \frac1y, \quad \quad \text{as}\quad
y \rightarrow \infty,$$ there is a constant $m>1$ such that whenever
$y>m$
$$ \frac12<\frac{\log(1+\frac1y) }{\frac1y} < \frac32.$$
Hence, for this constant $m$ we have
$$ A(x) = \int_{m}^{\infty} \frac{1}{x+y}\log\left(1+\frac{1}{y}
\right)\,dy \approx \int_{m}^{\infty} \frac{1}{x+y} \,\frac{dy}{y} =
\frac1x \log \frac{y}{x+y}\Big|_{m}^{\infty} \approx \frac{\log
x}{x}\,,
$$
which proves \eqref{log}.

Notice that the term $B(x)$ is better behaved :
$$ B(x) \le  \int_{2x+m}^{\infty} \frac{1}{y-x}\log\frac{y}{y-1} \,dy
\le \int_{2x+m}^{\infty} \frac2y \,\frac{dy}{y}
 \le \frac1x.
$$

\section{Proof of Theorem \ref{oteo} for first order Riesz transforms}
\label{Verdera:Sec:4}

In this Section we prove that
\begin{equation}\label{firstRiesz}
R_j^*(f)(x) \le C\,M^2(R_j(f)), \quad x \in \R^n,
\end{equation}
where $R_j$ is the $j$-th  Riesz transform, namely, the
Calder\'{o}n--Zygmund operator with kernel
$$
\frac{x_j}{|x|^{n+1}}, \quad x \in \mathbb{R}^n \setminus \{0\},
\quad 1 \leqslant j \leqslant n .
$$
Recall that $M^2 = M \circ M $ and notice that for $n=1$ we are
dealing with the Hilbert transform. The inequality
\eqref{firstRiesz} for the Hilbert transform is, as far as we know,
new. To have a glimpse at the difficulties we will encounter in
proving \eqref{firstRiesz} we start by discussing the case of the
Hilbert transform.

As in the even case we want to find a function $b$ such that
$$
\frac{1}{x} \chi_{\,\R \setminus (-1,1)}(x) = H(b).
$$
Since $H(-H)=I$
\begin{equation*}
\begin{split}
b(x) & = - H(\frac{1}{y} \chi_{\,\R \setminus (-1,1)}(y))
(x)\\*[5pt] & =\frac{1}{\pi} \int_{|y|>1}  \frac{1}{y-x}
\frac{1}{y}\,dy \\*[5pt] & = \frac{1}{\pi x}
\,\log\frac{|1+x|}{|1-x|}\,.
\end{split}
\end{equation*}
We conclude that, unlike in the even case, the function $b$ is
unbounded and is not supported in the unit interval $(-1,1).$ On the
positive side, we see that $b$ is a function in $BMO=BMO(\R),$ the
space of functions of bounded mean oscillation on te line. Since $b$
decays at infinity as $1/x^2$,\; $b$ is integrable on the whole
line. However, the minimal decreasing majorant of the absolute value
of $b$ is not integrable, owing to the poles at $\pm 1.$  This
prevents a pointwise estimate of $H^*f$ by a constant times $M(Hf).$
 We can now proceed with the proof of \eqref{firstRiesz} keeping in
mind the kind of difficulties we will have to overcome.

We start with the analog of Lemma \ref{nuc}. We denote by
$\operatorname{BMO}$ the space of functions of bounded mean
oscillation on $\R^n.$
\begin{lemma}\label{onuc}
There exists a function $b \in \operatorname{BMO}$ such that
\begin{equation}\label{nucfora}
\chi_{\R^n \setminus B}(x)\, \frac{x_j}{|x|^{n+1}} = R_j (b)(x),
\quad \text{for almost all} \quad x \in \R^n, \quad  1 \leqslant j
\leqslant n \, .
\end{equation}
\end{lemma}
\begin{proof}
For an appropriate constant $c_n$ the function
$$
E(x)= c_n\,\frac{1}{|x|^{n-1}}, \quad 0 \neq x \in \R^n
$$
satisfies
$$
\widehat{E}(\xi)= \frac{1}{|\xi|}, \quad 0 \neq \xi \in \R^n.
$$
Since the pseudo-differential operator  $(-\Delta)^{1/2}$ is defined
on the Fourier transform side as
$$
\widehat{ (-\Delta)^{1/2} \psi}(\xi) = |\xi| \hat{\psi}(\xi),
$$
$E$ may be understood as a fundamental solution of
$(-\Delta)^{1/2}.$ This will allow to structure our proof in
complete analogy to that of Lemma \ref{nuc} until new facts emerge.
Consider the function $\varphi$ that takes the value $c_n$ on $B$
and $E(x)$ on $\R^n \setminus B$ . We have that $\varphi = E
* (-\Delta)^{1/2} \varphi$ and we define $b$ as $(-\Delta)^{1/2}
\varphi.$

As it is well known
$$
\partial_j E = -(n-1) c_n \, p.v. \frac{x_j}{|x|^{n+1}} \,,
$$
in the distributions sense and, since $\varphi$ is continuous on the
boundary of $B$,
\begin{equation}\label{dfi}
\partial_j \varphi = -(n-1) c_n \chi_{\R^n \setminus B}(x)  \frac{x_j}{|x|^{n+1}}
\end{equation} also in the distributions sense. Then
\begin{equation*}
\begin{split}
-(n-1) c_n \chi_{\R^n \setminus B}(x)  \frac{x_j}{|x|^{n+1}} & =
\partial_j \varphi \\*[5pt] & = \partial_j E * b \\*[5pt] & = -(n-1) c_n \, p.v.
\frac{x_j}{|x|^{n+1}}* b,
\end{split}
\end{equation*}
which is \eqref{nucfora}. It remains to show that $b \in BMO.$

 Checking on the Fourier transform side we easily see that
\begin{equation}\label{b}
b= (-\Delta)^{1/2} \varphi = \gamma_n \,\sum_{k=1}^{n}
R_k(\partial_k \varphi),
\end{equation}
for some dimensional constant $\gamma_n.$  Since $\partial_k
\varphi$ is a bounded function by \eqref{dfi} and $R_k$ maps
$L^\infty$ into $\operatorname{BMO}$, $b$ is in $\operatorname{BMO}$
and the proof is complete. 
\end{proof}

Unfortunately $b$ is not bounded and is not supported on $\R^n
\setminus B.$ Moreover one can check easily that $b$ blows up at the
boundary of $B$ as the function $\log (1/|1-|x||).$  This entails
that the the minimal decreasing majorant of the absolute value of
$b$ is not integrable, as in the one dimensional case.

We take up now the proof of \eqref{firstRiesz}. By translation and
dilation invariance we only have to estimate the truncation of $R_j
f$ at the point $x=0$ and at level $\epsilon = 1$.  By Lemma
\ref{onuc}
\begin{equation*}\label{truncRiesz}
\begin{split}
 R_j^1 f(0) &=  - \int  \chi_{\R^n \setminus B}(x)\, \frac{x_j}{|x|^{n+1}} \,f(x)\, dx = - \int R_j b(x)\,f(x)\,dx\\*[5pt]
& = \int b(x) \, R_j f(x)\, dx\,.
\end{split}
\end{equation*}
Let $b_{2B}$ denote the mean of $b$ on the ball $2 B.$  We split the
last integral above into three pieces
\begin{equation}
\begin{split}
R_j^1 f(0) & = \int_{2B} (b(x)-b_{2B}) \, R_jf(x) \,dx + b_{2B}\,
\int_{2B} R_jf(x) \,dx + \int_{\R^n \setminus 2B} b(x) \, R_jf(x)
\,dx
\\*[5pt] & = I_1+I_2 +I_3 \,.
\end{split}
\end{equation}
Since $b_{2B}$ is a dimensional constant the term $I_2$ can be
immediately estimated by $C \, M(R_j f)(0).$   The term $I_3$ can
easily be estimated if we first prove that
\begin{equation}\label{decay}
|b(x)| \leqslant C \frac{1}{|x|^{n+1}},\quad |x|\geqslant 2\,.
\end{equation}
Indeed, the preceding decay inequality yields
$$
|I_3| \leqslant  C  \int_{\R^n \setminus 2B} |R_j f (x)|
\frac{1}{|x|^{n+1}}\, dx \leqslant C \, M(R_jf)(0)\,.
$$

To prove \eqref{decay} express $b$ by means of \eqref{b}
\begin{equation*}
\begin{split}
\frac{b}{\gamma_n}= \sum_{k=1}^n R_k \star \chi_{\R^n \setminus
B}(x)\, \frac{x_k}{|x|^{n+1}} & = \sum_{k=1}^n R_k \star R_k -
\sum_{k=1}^n R_k \star \chi_{B}(x) \frac{x_k}{|x|^{n+1}} \\ & =
\gamma'_n \, \delta_0 - \sum_{k=1}^n R_k (\chi_{B}(x)
\frac{x_k}{|x|^{n+1}})\,,
\end{split}
\end{equation*} where $\gamma'_n$ is a dimensional constant and $\delta_0$ the dirac delta at the origin.
The preceding formula for $b$ looks magical and one may even think
that some terms make no sense. For instance, the term $ R_k \star
R_k$ should not be thought as the action of the $k$-th Riesz
transform of the distribution $p.v.\, x_k/|x|^{n+1}$. It is more
convenient to look at it on the Fourier transform side, where you
see immediately that it is $\gamma'_n \, \delta_0$. The term $R_k
\star \chi_{B}(x) \frac{x_k}{|x|^{n+1}}$ should be thought as a
distribution, which acts on a test function as one would expect via
principal values (see below).

 If $|x|> 1$ we have
\begin{equation*}
\begin{split}
 R_k (\chi_{B}(x)\frac{x_k}{|x|^{n+1}})(x) &=  \lim_{\epsilon \rightarrow 0} \int_{\epsilon< |y|< 1}
\frac{x_k-y_k}{|x-y|^{n+1}} \frac{y_k}{|y|^{n+1}}\, dy \\*[5pt] &=
\lim_{\epsilon \rightarrow 0} \int_{\epsilon< |y|< 1}
\,\left(\frac{x_k-y_k}{|x-y|^{n+1}}-\frac{x_k}{|x|^{n+1}}
\right)\frac{y_k}{|y|^{n+1}}\, dy \,.
\end{split}
\end{equation*}
Since
$$
|\frac{x_k-y_k}{|x-y|^{n+1}}-\frac{x_k}{|x|^{n+1}}| \leqslant  C
\frac{|y|}{|x|^{n+1}}, \quad |x|\geqslant 2,\quad  |y|\leqslant 1
\,,
$$
we obtain, for $|x|\geqslant 2$,
$$
|R_k (\chi_{B}(x)\frac{x_k}{|x|^{n+1}})(x)|\leqslant  C
\int_{|y|<1} \frac{1}{|x|^{n+1}} \frac{1}{|y|^{n-1}} \,dy =
\frac{C}{|x|^{n+1}}\,,
$$
which gives \eqref{decay}.

We are left with the term $I_1.$  Since $b$ is in $BMO$ it is
exponentially integrable by John-Nirenberg's Theorem. We estimate
$I_1$ by Holder's inequality associated with the ``dual'' Young
functions $e^t-1$ and $t + t \log^+ t$ \,(\cite[p. 165]{GrMF}). We
get
$$
|I_1| \leqslant C\,\|b\|_{BMO} \|R_jf\|_{L\log{L}(2B)},
$$
where, for an integrable function g on $2B$,
$$ \|g\|_{L \log{L}(2B)}= \inf \{\lambda>0 \,:\, \frac{1}{|2B|} \int_{2B} \left
(\frac{|g(x)|}{\lambda} + \frac{|g(x)|}{\lambda} \log^+
(\frac{|g(x)|}{\lambda})\right )\,dx \leqslant 1 \}.
$$
It is a nice fact (see \cite{P} or \cite[p.159]{GrMF})  that the
maximal operator associated with $L\log L$, that is,
$$
M_{L(\log L)}g(x) = \sup_{Q \ni x}\|f\|_{L(\log L), Q},$$ the
supremum being over all balls $Q$, satisfies
\begin{equation}\label{iteration}
M_{L (\log{L})}f(x) \approx M^{2}f(x), \quad x \in \R^n.
\end{equation}
Thus
$$
|I_1| \le C\, M^2(R_j f)(0)
$$
and the proof of \eqref{firstRiesz} is complete.

\section{Necessary conditions for the $L^2$ estimate of $T^*f$ by $Tf$}
\label{Verdera:Sec:5} In this Section we find the necessary
conditions for the $L^2$ estimate
\begin{equation}\label{ela2}
 \| T^{\star} f  \|_2 \leqslant  C \| T f \|_2,\quad f \in L^2(\mathbb{R}^n)
\end{equation}
which are stated in (iii) of Theorem \ref{MOV} for the case of even
kernels. In particular, this will supply many even kernels for which
the preceding estimate fails (and thus the pointwise estimate in (i)
of Theorem \ref{MOV} fails).

We will look at the simplest possible situation. The kernel of our
operator $T$ in the plane is of the form
\begin{equation}\label{ke}
K(z)= -\frac{1}{\pi}\frac{xy}{|z|^4}+\frac{2}{\pi}
\frac{P_4(z)}{|z|^6},
\end{equation} where $z=x+i y$ is the complex variable in the plane $\C$ and
$P_4$ is a harmonic homogeneous polynomial of degree $4.$ The
constants in front of the two terms are set so that the expression
of the Fourier multiplier is the simplest. Indeed, the Fourier
transform of the principal value tempered distribution associated
with $K$ is
$$
\widehat{p.v.\,K}(\xi) = \frac{u v}{|\xi|^2}+
\frac{P_4(\xi)}{|\xi|^4}, \quad 0 \neq \xi =u+i v \in \C,
$$
by \cite[p.73]{St}. Our purpose is to find necessary conditions on
$P_4$ so that the $L^2$ estimate \eqref{ela2} holds. Notice that the
kernel $K$ is not harmonic, except in the case $P_4 =0$ which we
ignore. The spherical harmonics expansion of $K$ is reduced to the
sum of the two terms in \eqref{ke}.

Let $E$ be the standard fundamental solution of the bilaplacian
$\Delta^2$ in the plane. Thus
$$
E(z)= \frac{1}{8\pi}\,|z|^2\,\log|z|,\quad  \quad 0 \neq z \in \C,
$$
and $\hat{E}(\xi)= |\xi|^{-4}, \quad  0 \neq \xi \in \C.$  We have
$$
\left(\partial_1\partial_2 \Delta +P_4(\partial_1 ,\partial_2)
\right)(E) = p.v. K
$$
as one easily checks on the Fourier transform side. Here we adopt
the usual convention of denoting by $P_4(\partial_1 ,\partial_2)$ th
differential operator obtained by replacing the variables $x$ and
$y$ of $P_4$ by $\partial_1$ and  $\partial_2$ respectively.

Define a function $\varphi$ by
\begin{equation*}
\varphi(z)=\begin{cases}
E(z) & \quad \text{on } \mathbb{C} \setminus B\\
A_{0}+A_{1}|z|^2 + A_2 |z|^4  +A_3 |z|^6  & \quad \text{on }B
\end{cases}
\end{equation*}
where $B$ is the ball centered at the origin of radius $1.$ The
constants $A_j, 0 \le j \le 3,$ are chosen so that all derivatives
of $\varphi$ of order not greater than $3$ are continuous. This can
be done because $E$ is radial. With this choice to compute a fourth
order derivative of $\varphi$ in the distributions sense we only
need to compute the corresponding pointwise derivative of $\varphi$
in $B$ and on its complement. Set $b=\Delta^2 \varphi$, so that
$$
\varphi= E*\Delta^2 \varphi = E*b.
$$
A straightforward computation yields
$$
b= \Delta^2 \varphi = \chi_B(z) (\alpha+\beta |z|^2),
$$
for some constants $\alpha$ and $\beta.$  Then, as in the proof of
the $L^2$ estimate \eqref{ela2} for even second order Riesz
transforms presented in Section \ref{Verdera:Sec:2}, $b$ is
supported on the ball $B$ and is bounded. Set
$$
L = \partial_1\partial_2 \Delta+P_4(\partial_1 ,\partial_2),
$$
so that
$$
L(\varphi) = L(E)*b = p.v. K * b = T(b).
$$
On the other hand, by the definition of $\varphi$,
$$
L(\varphi) = \chi_{\C \setminus B}(z) K(z) + L(A_{0}+A_{1}|z|^2 +
A_2 |z|^4  +A_3 |z|^6 )\chi_B(z)\,.
$$
Now the term $ L(A_{0}+A_{1}|z|^2 + A_2 |z|^4  +A_3 |z|^6 )$ does
not vanish. Indeed, one can see that for some constant $c$
$$
 L(A_{0}+A_{1}|z|^2 + A_2 |z|^4  +A_3 |z|^6 ) = c\,x y \,.
$$
The result follows from the following three facts :
$$
 \left(\partial_1\partial_2 \Delta \right)(|z|^4) = 0,
$$
$$
 \left(\partial_1\partial_2 \Delta \right)(|z|^6) =  c\,x y
$$
and
$$
P_4(|z|^4) = P_4(|z|^6)=0.
$$
The last identity is due to the fact that $P_4$ is a homogeneous
harmonic polynomial of degree $4$. Notice that a priori $P_4(|z|^4)$
is a constant and $ P_4(|z|^6)$ is a homogeneous polynomial of
degree $2.$ The reader can verify that they are both zero just by
taking the Fourier transform and then checking their action on a
test function.

The conclusion is that
\begin{equation}\label{tebe}
T(b)= \chi_{\C \setminus B}(z) K(z) + c x y \chi_B(z).
\end{equation}
The novelty with respect to the argument of Section
\ref{Verdera:Sec:2} involving second order Riesz transforms is the
second term in the right hand side of the preceding formula.
Convolving \eqref{tebe} with a function $f$ in $L^2(\C)$ one gets
$$
c x y \chi_B(z) * f = T(f)* b - T^1(f),
$$
where $T^1 f$ is the truncation at level $1.$ Now, if \eqref{ela2}
holds then, since $b \in L^1(\C)$,
$$
\|c x y \chi_B(z) * f  \|_2 \le C \, \|T(f)\|_2, \quad f \in
L^2(\C),
$$
and hence, passing to the multipliers,
\begin{equation}\label{fu}
| \widehat{c x y \chi_B(z)} (\xi)| \le C\, \frac {|u v |\xi|^2+
P_4(\xi)|}{|\xi|^4}, \quad \xi \ne 0 .
\end{equation}
 Our next task is to understand the left hand side of  the above
inequality to obtain useful relations between the zero sets of the
various polynomials at hand. We should recall that the Fourier
transform of the characteristic function of the unit ball in $\R^2$
is $J_1(\xi)/|\xi|$, where $J_1(\xi)$ is the Bessel function of
order $1$. Write $G_m(\xi)= J_m(\xi)/|\xi|^m.$  The functions $G_m$
are radial and so we can view them as depending on a non-negative
real variable $r.$    We have \cite[p.425]{GrCF} the useful identity
$$
\frac{1}{r} \frac{d G_m}{dr}(r) = - G_{m+1}(r), \quad 0 \le r.
$$
From this it is easy to obtain the formula
\begin{equation*}
\begin{split}
\widehat{xy \chi_B(z)} (\xi) &=  - \partial_1 \partial_2
(G_1(|\xi|))
\\*[5pt] &= - u v \,G_3(|\xi|),
\end{split}
\end{equation*}
which transforms \eqref{fu} into
\begin{equation}\label{fu2}
| u v G_3(|\xi|)| \le C\, \frac {|u v |\xi|^2+ P_4(\xi)|}{|\xi|^4},
\quad \xi \ne 0 .
\end{equation}
Set
$$Q(\xi) = u v |\xi|^2+ P_4(\xi), \quad \xi \in \C .$$
Then \eqref{fu2} becomes, on the unit circle,
\begin{equation}\label{fu3}
| u v | \le C\,| Q(\xi)|, \quad  |\xi|=1.
\end{equation}
The above inequality encodes valuable information on the zero set of
$P_4.$ Recall that our goal is to show that $uv$ divides $P_4.$

Observe that $Q$ is a real polynomial with zero integral on the unit
circle, as sum of two non- constant homogeneous harmonic
polynomials. Thus $Q$ vanishes at some point $\xi=u+i v$ on the unit
circle. Then $ u v = 0$ by \eqref{fu3} and so $P_4(\xi)=0,$ owing to
the definition of $Q.$ We need now a precise expression for $P_4.$
The general harmonic homogeneous polynomial of degree $4$ is
\begin{equation}\label{qu}
\text{Re}(\lambda \xi^4) = \alpha (u^3 v - v^3 u)+ \beta (u^4+v^4-6
u^2 v^2),
\end{equation}
where $\lambda$ is a complex number and $\alpha$ and $\beta$ are
real. Assume that $P_4$ is as above. We know that $u^2+v^2 =1,$
$P_4(u,v)=0$ and that $u v=0.$ If $u=0$, then $\beta v^4 =0,$ which
yields $\beta=0.$ If $v=0$, then $\beta u^4 =0$ and we conclude
again that $\beta=0.$ Therefore
\begin{equation}\label{pe}
P_4(u,v)= \alpha (u^3 v - v^3 u)
\end{equation}
and $uv$ divides $P_4(u,v).$ We immediately conclude that the
operator $T$ with kernel
$$
K(z)= \frac{xy}{|z|^4}+ \frac{x^4+y^4-6 x^2 y^2}{|z|^6}, \quad 0 \ne
z \in \C,
$$
is an example in which the $L^2$ inequality \eqref{ela2} fails.
Before going on we remark that a key step in proving the division
property has been that $Q$ has at least one zero on the circle. This
is also a central fact in the proof of the general case.

We can easily deduce now another necessary condition for
\eqref{ela2} .  Substituting \eqref{pe} in \eqref{fu3} and
simplifying the common factor $u v$ we get
$$
0<|G_3(1)| \le C\, (1+ \alpha (u^2-v^2)), \quad |\xi|=1,
$$
which means that the right hand side cannot vanish on the unit
circle, namely, $|\alpha|< 1.$ Therefore we get the structural
condition
$$
T= R_P \circ U,
$$
where $R_P$ is the Riesz transform associated with the polynomial
$P(x,y)= -(1/\pi) \;x y$ and $U$ is an invertible operator in the
Calder\'{o}n--Zygmund algebra $A$.

Taking $\alpha=1$ we get an operator $T$ for which \eqref{ela2}
fails but whose kernel
$$
K(z)= -\frac{1}{\pi}\frac{xy}{|z|^4}+\frac{2}{\pi} \frac{x^3 y - x,
 y^3}{|z|^6}, \quad 0 \ne z \in \C,
$$
satisfies the division property of Theorem \ref{MOV} part (iii).

\section{Sufficient conditions for the $L^2$ estimate of $T^*f$ by $Tf$}
\label{Verdera:Sec:6}

In this Section we show how condition (iii) in Theorem \ref{MOV}
yields the pointwise inequality
\begin{equation}\label{pu}
 T^{*}f(z) \leqslant C \, M(Tf)(z),\quad z \in \C\,.
\end{equation}
As in the previous Section, we work in the particularly simple case
in which the spherical harmonics expansion of the kernel is reduced
to two terms. The first is a harmonic homogeneous polynomial of
degree $2$, which for definiteness is taken to be
$$
P(z)= -\frac{1}{\pi} x y.
$$
The second term is a fourth degree harmonic homogeneous polynomial.
The division assumption in (iii) of Theorem \ref{MOV} is that $P$
divides this second term. In view of the general form of a fourth
degree harmonic homogeneous polynomial \eqref{qu} we conclude that
our kernel must be of the form
$$
K(z)= -\frac{1}{\pi}\frac{xy}{|z|^4}+\frac{2}{\pi} \alpha  \frac{x^3
y - x y^3}{|z|^6}, \quad 0 \neq z \in \C, \quad \alpha \in \R.
$$

The second assumption in (iii) of Theorem \ref{MOV} is that $T$ is
of the form $T=R_P \circ U$, where $R_P$ is the second order Riesz
transform determined by $P$ and $U$ is an invertible operator in the
 Calder\'{o}n--Zygmund algebra $A.$ This is equivalent, as one can
 easily check looking at multipliers in the Fourier transform side,
  to  $|\alpha|<  1.$

In the simple context we have just set the two assumptions of
condition (iii) of Theorem \ref{MOV} are not independent. The reader
can easily check that the structural condition $T=R_P \circ U$
implies the division property, that is, that $P$ divides the fourth
degree term. We will point out later on where this simplifies the
argument.

We start now the proof of \eqref{pu}. Recall that, as we showed in
the preceding Section, there exists a bounded mesurable function $b$
supported on the unit ball $B$ and a constant $c$ such that
\begin{equation}\label{tebe2}
T(b)= \chi_{\C \setminus B}(z) K(z) + c x y \chi_B(z).
\end{equation}
Our goal is to express the second term in the right hand side above
as
\begin{equation}\label{beta}
c x y \chi_B(z) = T(\beta)(z), \quad \text{for almost all}\; z \in
\C,
\end{equation}
where $\beta$ is a bounded measurable function such that
\begin{equation}\label{de}
|\beta(z)| \leqslant \frac{C}{|z|^{3}},\quad |z|\geqslant 2\,.
\end{equation}
We first show that this is enough for \eqref{pu}. The only
difficulty is that $\beta$ is not supported in $B,$ but the decay
inequality  \eqref{de} is an excellent substitute. Set $\gamma =
b-\beta.$ Then ($dA$ is planar Lebesgue measure)
\begin{equation*}
\begin{split}
T^1 f(0) & = \int \chi_{\C \setminus B}(z )K(z)\, f(z) \,dA(z) \\
 & = \int T(\gamma)(z)\,f(z)\,dA(z) \\
 & = \int \gamma(z)\, Tf(z)\,dA(z) \\
 &= \int_{2B} \gamma(z) \, Tf(z)\,dz + \int_{\C \setminus 2B} \gamma(z) \,
 Tf(z)\,dA(z).
\end{split}
\end{equation*}
The first term is clearly less that a constant times $M(TF)(0),$
because $\gamma$ is bounded, and the second too, because of
\eqref{de} with $\beta$ replaced by $\gamma.$

The proof of \eqref{beta} is divided into two steps. The first step
consists in showing that there exists a function $\beta_0$ such that
$$
c x y \chi_B(z) = R(\beta_0)(z), \quad \text{for almost all}\; z \in
\C,
$$
where $R =R_P.$  To find $\beta_0$ let us look for a function $\psi$
such that
\begin{equation}\label{Verdera:psi}
P(\partial)\psi = c x y \chi_B(z).
\end{equation}
Assume that we have found $\psi$ and that it is regular enough so
that
$$
\psi= E* \Delta \psi,
$$
where $E$ is the standard fundamental solution of the Laplacian.
Then
\begin{equation*}
\begin{split}
c x y \chi_B(z) & = P(\partial)\psi = P(\partial)E \star \Delta
\psi\\*[5pt] & = c\, p.v.\,\frac{P(x)}{|z|^{4}} \star \Delta \psi=
R(\beta_0)\,,
\end{split}
\end{equation*} where $\beta_0 =c\, \Delta \psi.$

Taking the Fourier transform in \eqref{Verdera:psi} gives
$$
P(\xi)\hat{\psi}(\xi) = c \,\partial_1 \partial_2
\widehat{\chi_B}(\xi) = c\, u v \,G_3(|\xi|).$$ For the definition
of $G_3$ see the paragraph below \eqref{fu}. Hence
$$
\hat{\psi}(\xi)= c\, G_3(\xi),
$$
where $c$ is some constant. It is a well known fact in the
elementary theory of Bessel functions \cite[p.429]{GrCF} that
$$
c \,G_3(\xi)= \widehat{(1-|z|^2)^2 \chi_B(z)}(\xi)\,.
$$
In other words,
$$
\psi(z)= c \,(1-|z|^2)^2 \chi_B(z\,)\,.
$$
Clearly $\psi$ and its first order derivatives are continuous
functions supported on the closed unit ball $B.$ The second order
derivatives of $\psi$ are supported on $B$ and on $B$ they are
polynomials.  In particular, we get that $\beta_0 = c\, \Delta\psi $
is a function supported on $B,$ which satisfies a Lipschitz
condition on $B$ and satisfies the cancellation property $\int
\beta_0 = c\, \int \Delta \psi = 0$.

It is worth remarking that in the general case, where the spherical
harmonic expansion of the kernel contains many terms, one has to
resort to the division assumption of (iii) in Theorem \ref{MOV} to
complete the proof of the first step.

We proceed now with the second step. Since $T = R \circ U$ we have
$$
c\, x y \chi_B(z) = R(\beta_0)(z) = T(U^{-1}(\beta_0))(z).
$$
Set $\beta=U^{-1}(\beta_0),$ so that \eqref{beta} is satisfied. We
are left with the task of showing that $\beta$ is bounded and
satisfies the decay estimate \eqref{de}.

The inverse of $U$ is an operator in the Calder\'{o}n--Zygmund
algebra $A.$ Thus
$$
\beta= U^{-1}(\beta_0) = (\lambda I + V)(\beta_0) = \lambda \beta_0
+ V(\beta_0),
$$
where $\lambda$ is a real number and $V$ an even convolution smooth
homogeneous Calder\'{o}n--Zygmund operator. The desired decay
estimate for $\beta$ now follows readily, because $\beta_0$ is
supported in the closed ball $B$ and has zero integral. It remains
to show that $V(\beta_0)$ is bounded. At first glance this is quite
unlikely because $V$ is a general even convolution smooth
homogeneous Calder\'{o}n--Zygmund operator and $\beta_0$ has no
global smoothness properties in the plane. Indeed, although
$\beta_0$ is Lipschitz on $B$, it has a jump at the boundary of $B.$
Assume for a moment that $\beta_0 = \chi_B.$  It is then known that
$V(\chi_B)$ is a bounded function because $V$ is an even
Calder\'{o}n--Zygmund operator and the boundary of $B$ is smooth.
Here the fact that the operator is even is crucial as one can see by
considering the action of the Hilbert transform on the interval
$(-1,1).$ We are not going to present the nice argument for the
proof that $V(\beta_0)$ is bounded \cite{MOV}. Let us only mention
that this result for the Beurling transform and smoothly bounded
domains plays a basic role in the regularity theory of certain
solutions of the Euler equation in the plane \cite{Ch}.

\section{The proof in the general case and final comments}
\label{Verdera:Sec:7}

The proof of Theorems \ref{MOV} and \ref{MOPV} in the general case
proceeds in two stages. First one proves the Theorems in the case in
which the spherical harmonics expansion of the kernel contains
finitely many non-zero terms. Then one has to truncate the expansion
of the kernel and see that some of the estimates obtained in the
first step do not depend on the number of terms. This is a delicate
issue at some moments, but necessary to perform a final compactness
argument. In both steps there are difficulties of various types to
be overcome and a major computational issue, lengthly and involved,
which very likely can be substantially simplified by a more clever
argument.

A final word on the proof for the necessity of the division
condition. To show that a polynomial with complex coefficients
divides another, one often resorts to Hilbert's Nullstellensatz, the
zero set theorem of Hilbert, which states that if $P$ is a prime
polynomial with complex coefficients and finitely many variables, to
show that $P$ divides another such polynomial $Q$ one  has to check
only that $Q$ vanishes on the zeros of $P.$ This fails for real
polynomials, as simple examples show. Now, since we are working with
real polynomials we cannot straightforwardly apply Hilbert's
theorem. What saves us is that our real polynomials have a fairly
substantial amount of zeroes, just because they have zero integral
on the unit sphere. We can then jump to the complex case and come
back to the real by checking that the Hausdorff dimension of the
zero set of certain polynomials is big enough.

There are several questions about Theorems \ref{MOV} and \ref{MOPV}
that deserve further study. The first is a potential application to
the David--Semmes problem mentioned in the introduction, which was
the source of the question. Another is the smoothness of the
kernels. It is not known how  to prove the analogs of Theorems
\ref{MOV} and \ref{MOPV} for kernels of moderate smoothness, say of
class $C^m$ for some positive integer $m.$ Finally it is has
recently been shown by Bosch, Mateu and Orobitg that
$$
 \| T^{\star} f  \|_p \leqslant  C \| T f \|_p,\quad f \in L^p(\mathbb{R}^n)
 \,, \quad 1 <  p < \infty\,,
$$
implies any of the three equivalent conditions in Theorems \ref{MOV}
and \ref{MOPV}.

\begin{gracies}
The author is grateful to the organizers of the XXXI Conference in
Harmonic Analysis, held in Rome at the beginning of June 2011 in
honor of Professor A.Fig\`{a}-Talamanca, for the kind invitation to
participate and for their efficient work. The author has been
partially supported by the grants MTM2010-15657 and 2009SGR420.
\end{gracies}

\begin{tabular}{l}
Joan Verdera\\
Departament de Matem\`{a}tiques\\
Universitat Aut\`{o}noma de Barcelona\\
08193 Bellaterra, Barcelona, Catalonia\\
{\it E-mail:} {\tt jvm@mat.uab.cat}
\end{tabular}

\end{document}